\documentclass[12pt, reqno]{amsart}

\usepackage[
backend=biber,
style=numeric-comp,
sorting=none
]{biblatex}
\AtBeginBibliography{\scriptsize}
\usepackage{amssymb}
\usepackage{algpseudocode}
\usepackage{algorithm}
\usepackage{verbatim}
\usepackage{graphicx}
\usepackage{placeins}
\usepackage{listings}
\usepackage{xcolor}
\usepackage{subcaption}
\usepackage{enumitem}
\usepackage{amsmath}

\RequirePackage{dsfont} 
 \scrollmode
\allowdisplaybreaks
\usepackage{graphicx}
\usepackage{epstopdf}
\usepackage{hyperref}
\usepackage[T1]{fontenc}

\usepackage{amsmath}
\usepackage{tikz}

\makeatletter
\newcommand{\xleftrightarrow}[2][]{\ext@arrow 3359\leftrightarrowfill@{#1}{#2}}
\makeatother

\definecolor{codegreen}{rgb}{0,0.6,0}
\definecolor{codegray}{rgb}{0.5,0.5,0.5}
\definecolor{codepurple}{rgb}{0.58,0,0.82}
\definecolor{backcolour}{rgb}{0.95,0.95,0.92}

\lstdefinestyle{list_style}{
  backgroundcolor=\color{backcolour}, commentstyle=\color{codegreen},
  keywordstyle=\color{magenta},
  numberstyle=\tiny\color{codegray},
  stringstyle=\color{codepurple},
  basicstyle=\ttfamily\footnotesize,
  breakatwhitespace=false,         
  breaklines=true,                 
  captionpos=b,                    
  keepspaces=true,                 
  numbers=left,                    
  numbersep=5pt,                  
  showspaces=false,                
  showstringspaces=false,
  showtabs=false,                  
  tabsize=2
}

\newcommand{\xdasharrow}[2][->]{
\tikz[baseline=-\the\dimexpr\fontdimen22\textfont2\relax]{
\node[anchor=south,font=\scriptsize, inner ysep=1.5pt,outer xsep=2.2pt](x){#2};
\draw[shorten <=3.4pt,shorten >=3.4pt,dashed,#1](x.south west)--(x.south east);
}
}

\newcommand{\DEBUG}{}

\ifdefined\DEBUG%

  \def\rem#1{{\marginpar{\raggedright\scriptsize #1}}}
  \newcommand{\pmr}[1]{\rem{\color{blue}{$\bullet$ #1}}}
  \newcommand{\ppr}[1]{\rem{\color{red}{$\bullet$ #1}}}
 \else%

  \newcommand{\ppr}[1]{}
  \newcommand{\pmr}[1]{}
 \fi

\newcommand{\one}{\cdot\mathds{1}}

\newcommand{\bigo}{\mathcal{O}}
\newcommand{\var}
{\operatorname{Var}}

\newcommand{\cost}{\operatorname{cost}}

\newcommand{\calf}{{\mathcal F}}

\newcommand{\euler}{e}

\def\rho{\varrho_1}

\def\rd{\,{\mathrm d}}
\def\mc{\mathcal{MC}}
\def\mlmc{\mathcal{ML}}

\theoremstyle{plain}
\newtheorem{theorem}{Theorem}

\newtheorem{definition}{Definition}
\newtheorem{fact}{Fact}

\newtheorem{proposition}{Proposition}

\theoremstyle{definition}
\newtheorem{remark}{Remark}
\newtheorem*{example}{Example}

\begin{document}

\title
[Multilevel Monte Carlo]
{A multilevel Monte Carlo algorithm for SDEs driven by countably dimensional Wiener process 
and Poisson random measure
}

\author[M. Sobieraj]{Micha{\l} Sobieraj}
\address{AGH University of Krakow,
	Faculty of Applied Mathematics,
	Al. A.~Mickiewicza 30, 30-059 Krak\'ow, Poland}
\email{sobieraj@agh.edu.pl, corresponding author}

\begin{abstract}
In this paper, we investigate properties of standard and multilevel Monte Carlo methods for weak approximation of solutions of stochastic differential equations (SDEs) driven by infinite-dimensional Wiener process and Poisson random measure with Lipschitz payoff function. The error of the truncated dimension randomized numerical scheme, which depends on two parameters, i.e grid density $n \in \mathbb{N}$ and truncation dimension parameter $M \in \mathbb{N},$ is of the order $n^{-1/2}+\delta(M)$ such that $\delta(\cdot)$ is positive and decreasing to $0$. We derive complexity model and provide proof for the upper complexity bound of the multilevel Monte Carlo method which depends on two increasing sequences of parameters for both $n$ and $M.$ The complexity is measured in terms of upper bound for mean-squared error and is compared with the complexity of the standard Monte Carlo algorithm. The results from numerical experiments as well as Python and CUDA C implementation are also reported. 
\end{abstract}
\maketitle


\section{Introduction}
For $T>0,$ we investigate the problem of efficient approximation of 
$$\mathbb{E}(f(X(T)))$$ 
for a unique strong solution $(X(t))_{t \in [0, T]}$ of a system of $d \in \mathbb{N}$ stochastic differential equations that is written in the following form
\begin{equation}
\label{main_equation}
	\left\{ \begin{array}{ll}
	\displaystyle{
	\rd X(t) = a(t,X(t))\rd t + b(t,X(t)) \rd W(t)}
 + \int\limits_{\mathcal{E}}c(t,X(t-),y)N(\rd y,\rd t), \ t\in [0,T],\\
	X(0)=\eta,
	\end{array} \right.
\end{equation}
such that
$a$ and $c$ are $\mathbb{R}^{d}$-valued, $b$ takes values in the space of square-summable $\mathbb{R}^{d}$-valued sequences, $f: \mathbb{R}^{d} \mapsto \mathbb{R}$ is Lipschitz payoff function, $W = [W_1, W_2, \ldots]^T$ is countably dimensional Wiener process and $N$ is Poisson random measure.
Moreover,
\begin{equation*}
    \int\limits_{0}^{t}b(s,X(s))\rd W(s)=\sum\limits_{j=1}^{+\infty}\int\limits_0^t b^{(j)}(s,X(s))\rd W_{j}(s)
\end{equation*}
is the stochastic It\^o integral wrt the countably dimensional Wiener process $W$ (see pages 427-428 in \cite{cohen_elliott}).
Furthermore, we assume that $d' \in \mathbb{N},
\mathcal{E} := \mathbb{R}^{d'}\setminus \{0\}$ and the intensity measure of $N$ is $\nu(\rd y)\rd t$, where $\nu(\rd y)$ is a finite L\'{e}vy measure on $(\mathcal{E},\mathcal{B}(\mathcal{E}))$. We assume that both $N$ and $W$ are defined on the same complete probability space $(\Omega, \Sigma, \mathbb{P})$ and independent. 
Finally, we impose suitable regularity conditions on the coefficients $a,b,c$ and $\eta.$ 
Analytical properties and applications of such SDEs are widely investigated in \cite{cohen_elliott} and
\cite{Gyngy1980OnSE}. 

The infinite-dimensional Wiener process is the natural extension of standard finite-dimensional Brownian motion which allows us to model more complex structures of the underlying noise.
If $W$ is countably dimensional the stochastic It\^o integral can be understood as a stochastic integral wrt cylindrical Wiener process in the Hilbert space $\ell^2$, see pages 289-290 in \cite{cohen_elliott}.
For relation between theory of Stochastic Partial Differential Equations (SPDEs) and SDEs driven by countably dimensional Wiener, see \cite{walsh} and \cite{liu2015stochastic}.

In many cases, the existence and uniqueness of solutions of SDEs are guaranteed but
the closed-form formulas are not known. It leads to the usage of numerical schemes for approximation of trajectories.
In \cite{SIAM} authors introduce the truncated dimension
Euler algorithm for a strong pointwise approximation of the solutions of \eqref{main_equation} and provide its upper error bounds. The results are further used in this paper in the context of cost and error analysis for both standard and multilevel method.

In 2001, the multilevel Monte Carlo (MLMC) approach was first introduced by Stefan Heinrich (see  \cite{STEFAN}) in the context of parametric integration. Next, in 2008 Mike Giles applied the multilevel method to weak approximation problem in the context of SDEs (see \cite{GILES}). For now, there is a vast literature addressing the application of MLMC method to various classes of SDEs. Nonetheless, so far there was no preceding works that directly address the investigation of MLMC for SDEs driven by a countably dimensional Wiener process. On the other hand, the investigation of the MLMC method for SPDEs is very popular and, as mentioned, related to the concept of SDEs driven by a countably dimensional Wiener process.
For papers regarding SPDEs and MLMC method, the reader is especially referred to \cite{chada2022}, \cite{CLIFFE},
\cite{NEUMULLER}, 
\cite{SPDES1}, \cite{SPDES2}, \cite{SPDES3}, \cite{SPDES4} and \cite{SPDES5}. 

For application of MLMC to SDEs where coefficients of diffusion depend on the distribution of the solution itself, i.e McKean–Vlasov SDEs, the reader is referred to \cite{rached2023doubleloop} and \cite{Haji-Ali2018}.

From practical point of view, one of the MLMC method applications is its extensive usage in Finance (see \cite{gerstner_kloeden},
\cite{Giles2009MultilevelMC}, 
\cite{belomestny},
\cite{burgos_giles}).

To our best knowledge, MLMC can further be extended to Multi-Index Monte Carlo (MIMC) (see \cite{hajiali2015multiindex}). Nevertheless, it is a subject for future research because this geos beyond the scope of this paper.

The main contribution of this paper is a derivation of the cost model for weak approximation with a random number of evaluations and analysis of the cost bounds for standard and MLMC methods for SDEs driven by countably dimensional Wiener process (which induces an additional set of parameters in MLMC method) and Poisson random measure (which imposes the expected complexity model). Extension of the multilevel approach to SDEs driven by countably dimensional Wiener process and Poisson random measure is motivated by and analogous to the approach presented in \cite{GILES}.

The structure of the paper is as follows. In Section \ref{sec:preliminaries} we describe the considered class of SDEs \eqref{main_equation} with admissible coefficients defining the equation. We further recall recent results on the numerical scheme for a strong pointwise approximation of the solutions.
In Section \ref{sec:complex} we define the complexity model.
In Section \ref{sec:mc_algorithm} we investigate properties of the standard Monte Carlo algorithm in a defined setting.
In Section \ref{sec:mlmc_algorithm}
we derive the MLMC algorithm and provide a theorem which addresses its upper complexity bounds together with proof.  Finally, our theoretical results are supported by numerical experiments described in
Section \ref{sec:numer}. Therefore, we also provide the key elements of our current algorithm implementation
in Python and CUDA C.
In section \ref{sec:conclude} we summarize the main results and list the resulting open questions. 

\section{Preliminaries}\label{sec:preliminaries}
We first introduce basic notations and recall class of SDEs for which the strong pointwise approximation problem was investigated in \cite{SIAM}.

Let $x \wedge y := \min\{x, y\}$ and let $x \vee y := \max\{x, y\}$ for any $x,y \in \mathbb{R}.$
We use the following notation of asymptotic equalities. For functions $f,g: [0,+\infty) \to [0,+\infty)$ we write $ 	f(x)=\mathcal{O}(g(x))$ iff there exist $C>0, x_0>0$ such that for all $x \geq x_0$ it holds $f(x)\leq Cg(x)$. Furthermore, we write $f(x)=\Theta(g(x))$ iff $f(x)=\mathcal{O}(g(x)) \ \hbox{and} \ g(x)=\mathcal{O}(f(x))$. The preceding definitions can naturally be extended to arbitrary accumulation points in $[0, +\infty).$ 

Depending on the context, by $\| \cdot\|$ we either denote the Euclidean norm for vectors or Hilbert-Schmidt norm for matrices. The difference should be clear from the context.
We also set 
\begin{displaymath}
   \ell^2 (\mathbb{R}^d) = \{x = (x^{(1)}, x^{(2)}, \ldots)\ | \ x^{(j)}\in \mathbb{R}^d \ \hbox{for all} \ j\in\mathbb{N}, \|x\| < +\infty\},
\end{displaymath}
where $x^{(j)} = \begin{bmatrix} 
x_{1}^{(j)}\\
\vdots \\
x_{d}^{(j)} 
\end{bmatrix}$, $\displaystyle{\|x\|=\Bigl(\sum\limits_{j=1}^{+\infty}\|x^{(j)}\|^2\Bigr)^{1/2}=\Bigl(\sum\limits_{j=1}^{+\infty}\sum\limits_{k=1}^d|x_{k}^{(j)}|^2\Bigr)^{1/2} }$.

Let $(\Omega,\Sigma,\mathbb{P})$ be a complete probability space with sufficiently rich filtration $(\Sigma_{t})_{t\geq 0}$ that satisfies the usual conditions (see \cite{Protter}), and let $\Sigma_{\infty}:=\sigma\Bigl(\bigcup_{t\geq 0}\Sigma_t\Bigr).$ For any random vector ${X: \Omega \mapsto \mathbb{R}^{d}}$ we define its $L^{2}(\Omega)$ norm as $\|X\|_{L^{2}(\Omega)} := (\mathbb{E}\| X \|^{2})^{1/2},$ and by $X^{(i)}$ we mean the $i'th$ independent realization of the vector. Finally, for any random element, including random variable, random vector or random matrix, by $\sigma(X)$ we denote a $\sigma$-algebra generated by $X.$ Similarly, for a sequence of random elements $(X_{i})_{i=1}^{n},$ we define $\sigma$-field generated by it as $\sigma(X_1,X_2,\ldots,X_n):= \sigma\Bigl(\bigcup\limits_{1 \leq i \leq n}\sigma(X_{i})\Bigr).$

Let $d' \in \mathbb{N},
\mathcal{E} := \mathbb{R}^{d'}\setminus \{0\}$ and let $\nu$ be a L\'{e}vy measure on $(\mathcal{E},\mathcal{B}(\mathcal{E}))$, i.e., $\nu$ is a measure on $(\mathcal{E},\mathcal{B}(\mathcal{E}))$ that satisfies condition ${\displaystyle{\int\limits_{\mathcal{E}}(\|z\|^2 \wedge  1) \nu(\rd z)<+\infty}}.$ We further assume that $\lambda := \nu(\mathcal{E}) < +\infty$. 
Hence, by Theorem 1.4.1 in \cite{Kunita}, there exists a scalar  Poisson process $N=(N(t))_{t\geq 0}$ with intensity $\lambda$ and an iid sequence of $\mathcal{E}$-valued random variables $(\xi_k)_{k=1}^{+\infty}$ with the common distribution $\nu(\rd y)/\lambda$ such that the Poisson random measure $N(\rd z,\rd t)$ can be written as follows
$\displaystyle{N(E\times (s,t])=\sum\limits_{N(s) < k\leq N(t)}\mathbf{1}_E(\xi_k)}$ for $0\leq s<t\leq T,  E\in\mathcal{B}(\mathcal{E})$.
Furthermore, we assume that $(\Sigma_t)_{t \geq 0}$ is rich enough, such that $W$ is countably dimensional $(\Sigma_{t})_{t\geq 0}-$Wiener process, $N(\rd z,\rd t)$ is $(\Sigma_t)_{t\geq  0}$-Poisson random measure with the intensity measure $\nu(\rd z)\rd t$, and finally both $W$ and $N$ are independent.

For $D,D_L>0$ we consider $\mathcal{A}(D,D_L)$ a class of all functions $a: [0,T]\times \mathbb{R}^d \mapsto \mathbb{R}^d $ satisfying the following conditions:
\begin{enumerate}
	\item [(A1)] $a$ is Borel measurable,
	\item [(A2)] $\|a(t,0)\|\leq D$ for all $t\in[0,T]$,
	\item [(A3)] $\|a(t,x) - a(t,y)\| \leq D_L\|x-y\|$ for all $x,y \in \mathbb{R}^d$, $t \in [0,T]$.
\end{enumerate}
Let $\Delta = (\delta(k))_{k = 1}^{+\infty} \subset \mathbb{R}_+$ be a positive, strictly decreasing sequence, converging to zero, and let $C>0$,  $\varrho_1 \in (0,1]$.  We consider the following class $\mathcal{B}(C,D,D_L,\Delta,\varrho_1)$ of functions $b = (b^{(1)}, b^{(2)}, \ldots):[0,T] \times \mathbb{R}^d\mapsto\ell^2 (\mathbb{R}^d)$, where $\ b^{(j)}:[0,T] \times \mathbb{R}^d  \mapsto \mathbb{R}^d$, $j\in\mathbb{N}$. Namely, $b\in \mathcal{B}(C,D,D_L,\Delta,\varrho_1)$ iff it satisfies the following conditions:
\begin{enumerate}
	\item [(B1)] $\Vert b(0,0) \Vert \leq D$,
	\item [(B2)] $\Vert b(t,x) - b(s,x)\Vert \leq D_L(1+ \|x\|)|t-s|^{\varrho_1}$ for all $x \in \mathbb{R}^d$ and $t,s\in [0,T]$,
	\item [(B3)] $\Vert b(t,x) - b(t,y)\Vert \leq D_L \|x-y\|$	for all $x,y \in \mathbb{R}^d$ and $t \in [0,T]$,
	\item [(B4)] $\sup_{0\leq t \leq T}\Vert \sum_{i=k+1}^{+\infty} b^{(i)}(t,x)\Vert \leq C(1+ \|x\|)\delta(k)$ for all $k \in \mathbb{N}$ and $x\in \mathbb{R}$.
\end{enumerate}
By $\delta$ we also denote a function on $[1, +\infty)$ which is defined either by linear interpolation of $\Delta$ sequence or simple substitution of index $k$ with continuous variable $x$ in the definition. Such function is invertible and the difference between each $\delta$ is clear from the context.
Let $\varrho_{2}\in (0,1]$ and let $\nu$ be the L\'{e}vy measure as above. We say that a function $c: [0,T]\times \mathbb{R}^d \times \mathbb{R}^{d'} \mapsto \mathbb{R}^d$ belongs to the class $\mathcal{C}(D,D_L,\varrho_2, \nu)$ if and only if
\begin{enumerate}
    \item [(C1)] $c$ is Borel measurable,
	\item [(C2)] $\displaystyle{\biggr(\int\limits_{\mathcal{E}}\|c(0,0,y)\|^p \nu(\rd y)\biggr)^{1/2} \leq D}$,
	\item [(C3)] $\displaystyle{\biggr(\int\limits_{\mathcal{E}}\|c(t,x_1,y) - c(t,x_2,y)\|^2 \  \nu(\rd y)\biggr)^{1/2} \leq D_L\|x_1-x_2\|}$ for all $x_1,x_2 \in \mathbb{R}^d$, $t \in [0,T]$,
	\item [(C4)] $\displaystyle{\biggr(\int\limits_{\mathcal{E}}\|c(t_1,x,y) - c(t_2,x,y)\|^2  \ \nu(\rd y)\biggr)^{1/2} \leq D_L(1+\|x\|)|t_1-t_2|^{\varrho_2}}$ for all $x\in \mathbb{R}^d$, $t_1,t_2 \in [0,T]$.
\end{enumerate}
Finally, we define the following class of initial values
\begin{equation*}
	\mathcal{J}(D) = \{\eta \in L^2(\Omega) \ |  \ \sigma(\eta)\subset\Sigma_0, \Vert \eta \Vert_{L^{2}(\Omega)}\leq D\}.
\end{equation*}
As a set of admissible input data, we consider the following class
\begin{equation*}
	\calf(C,D,D_L,\Delta,\varrho_1, \varrho_2, \nu) = \mathcal{A}(D,D_L) \times \mathcal{B}(C,D,D_L,\Delta, \varrho_1) \times \mathcal{C}(D,D_L,\varrho_2, \nu) \times \mathcal{J}(D).
\end{equation*}

Next, we recall the truncated dimension randomized Euler algorithm, defined in \cite{SIAM} to approximate the value of $X(T)$. 
Note that the drift coefficient is only Borel measurable with respect to time variable, and thus not necessarily continuous. Therefore, the randomization is required to guarantee algorithm convergence for any input from $\mathcal{F}(C,D,D_L,\Delta, \varrho_1, \varrho_2, \nu)$. See section 3 in \cite{PRZYB2014} for the example that shows lack of convergence if someone uses standard Euler scheme.  

Let $M,n \in \mathbb{N}$,  $t_j = jT/n$, $j=0,1,\ldots, n$. We  also use the notation 
 $\Delta W_{j,k} = W_k(t_{j+1}) - W_k(t_j)$ for $k\in 0,1,\dots, M-1.$ Let $\left(\theta_j\right)_{j=0}^{n-1}$ be a sequence of independent random variables, where each $\theta_j$ is uniformly distributed on $[t_j,t_{j+1}]$, $j=0,1,\ldots,n-1$. We also assume that $\displaystyle{\sigma(\theta_0,\theta_1,\ldots,\theta_{n-1})}$ is  independent of $\Sigma_\infty$. For $(a,b,c,\eta) \in \calf( C,D,D_L,\Delta,\varrho_1, \varrho_2, \nu)$ we set
\begin{equation}\label{main_scheme}
	\begin{cases}
		X_{M,n}^{RE}(0) = \eta \\
		X_{M,n}^{RE}(t_{j+1}) = X_{M,n}^{RE}(t_{j}) + a(\theta_j , X_{M,n}^{RE}(t_{j}))\frac{T}{n} + \sum\limits_{k=1}^{M} b^{(k)}(t_j,X_{M,n}^{RE}(t_{j}))\Delta W_{j,k}\\  \quad\quad\quad\quad\quad\quad 
  +\sum\limits_{k=N(t_j)+1}^{N(t_{j+1})}c(t_j,X_{M,n}^{RE}(t_j), \xi_k), \quad j=0,1,\ldots, n-1
	\end{cases}.
\end{equation}
In the following part we assume that if the argument is omitted, then by $X^{RE}_{M,n}$ we mean the random vector $X_{M,n}^{RE}(T).$
By $X^{RE, i}_{M,n},$ we denote the $i'th$ independent sample of the random vector $X^{RE}_{M,n}.$

For the sake of brevity, we also recall the following theorem which corresponds to the rate of convergence of the presented algorithm. The more general case where error is measured in $L^p(\Omega)$ norm can be found in \cite{SIAM}.

\begin{theorem}[\cite{SIAM}]
\label{upper_bound}
	 There exists a constant $\kappa > 1$, depending only on the parameters of the class $\calf(C,D,D_L,\Delta,\varrho_1, \varrho_2, \nu)$, such that for every $(a,b,c, \eta)\in\calf( C,D,D_L,\Delta,\varrho_1, \varrho_2, \nu)$ \linebreak and $M,n\in\mathbb{N}$ it holds
	\begin{equation*}
	    \|X(T)- X^{RE}_{M,n}\|_{L^{2}(\Omega)}\leq \kappa\Bigl( n^{-\alpha} +  \delta(M)\Bigr)
	\end{equation*}
 where $\alpha:=\min\{\varrho_1, \varrho_2, 1/2\}.$
\end{theorem}
In the remaining part of this paper, by $(X(t))_{t\in [0,T]}$ we mean the unique strong solution of the equation \eqref{main_equation} that implicitly depends on $(a,b,c,\eta) \in \calf(C,D,D_L,\Delta,\varrho_1, \varrho_2, \nu).$
In the following section, we define a complexity model in terms of information-based complexity framework (see \cite{ibcbook}). 

\section{Complexity model}
\label{sec:complex}
In \cite{DEREICH20111565}, authors investigate L\'evy driven SDEs where random discretization of time grid imposes complexity measured in terms of expectation. Regarding the truncated dimension randomized Euler algorithm, we utilize the same approach, since the number of evaluations of $c$ is non-deterministic.

First, note that the initial value is $d$-dimensional random vector which is evaluated only once, and therefore its informational cost is always equal to $d.$
On the other hand, the number of scalar evaluations of $a$ in the truncated dimension randomized Euler algorithm is equal to $dn,$ since the algorithm evaluates $d$ coordinates of $a$ once per each step, and the number of steps is $n.$ In a similar manner, we obtain that the overall number of evaluations of $b$ is $dMn.$  
Nevertheless, there are $\sum\limits_{j=0}^{n-1}(N(t_{j+1})-N(t_{j})) = N(T)$ evaluations of $c,$ which results in informational cost of scalar evaluations that is equal to $N(T)d.$
Finally, there are $n$ evaluations of $M$ scalar Wiener increments and $n$ evaluations of Poisson point process $N,$ which results in informational costs of $Mn$ and $n$ respectively. 
Together with that, we arrive at the following proposition.

\begin{proposition}
The informational $\cost(X_{M,n}^{RE,i})$
 of the evaluation of a single random sample $X_{M,n}^{RE, i}$ is
\begin{equation*}
\begin{split}
\cost(X_{M,n}^{RE,i}) & := \#\{ \mbox{scalar evaluations of } a,b,c,\eta,W,N \} \\
& =
d(n + Mn + N(T) + 1) + Mn + n.
\end{split}
\end{equation*}
On the other hand, the expected informational cost of the algorithm is
\begin{equation*}
\begin{split}
    \cost \big{(} X_{M,n}^{RE}\big{)} & := \mathbb{E}\Big{[}\mbox{\# of scalar evaluations of } a,b,c,\eta,W,N \Big{]} \\
    & = d(n + Mn + \lambda T + 1) + Mn + n.
\end{split}
\end{equation*}
\end{proposition}
In this paper, we focus on the informational cost of the evaluation of many independent samples from $X_{M,n}^{RE}.$
From the law of large numbers, which applies to the sequence of i.i.d random variables $(\cost(X_{M,n}^{RE,i}))_{i=1}^{K},$ we obtain 
\begin{equation*}
\sum_{i=1}^{K}\cost(X_{M,n}^{RE,i}) = K (\frac{1}{K}\sum_{i=1}^{K}\cost(X_{M,n}^{RE,i})) \approx K \cost(X_{M,n}^{RE}).
\end{equation*}
Furthermore, we can construct asymptotic confidence interval for the aforementioned estimate. Let $\Phi^{-1}$ denote an inverse CDF of a Normal distribution. From central limit theorem, we obtain that
\begin{equation}
P\Big{(}\big{|}\cost(X_{M,n}^{RE}) -\frac{1}{K}\sum\limits_{i=1}^{K}\cost(X_{M,n}^{RE, i})\big{|} \leq \Phi^{-1}\big{(}(\alpha+1)/2\big{)}\frac{\lambda T}{\sqrt{K}}\Big{)} \approx \alpha
\end{equation}
for sufficiently large number of samples $K$ and confidence level $\alpha \in (0,1).$
Note that informational cost takes values in $\mathbb{N}.$ To guarantee that estimation error of random cost with expected cost is less than or equal to $1$ with probability $\alpha,$ it suffices to use at least $ K \geq \Big{(}\Phi^{-1}\big{(}(\alpha+1)/2\big{)}\lambda T\Big{)}^{2}$ samples. Conversely, for $K$ samples we obtain 
estimation error equal to $1$ with probability $2\Phi(\frac{\sqrt{K}}{\lambda T})-1.$

On the other hand, note that
\begin{equation*}
d Mn \leq \cost \big{(} X_{M,n}^{RE}\big{)} \leq d(\lambda T + 5) M n,
\end{equation*}
since $Mn \geq 1,$ which means that $ \cost \big{(} X_{M,n}^{RE}\big{)} = \Theta(Mn).$
Hence, we define the overall cost of the evaluation of i.i.d. samples 
$(X_{M,n}^{RE,i})_{i=1}^{K}$ as
\begin{equation*}
\cost ( (X_{M,n}^{RE,i})_{i=1}^{K} ) := KMn.
\end{equation*}

\section{Monte Carlo method}\label{sec:mc_algorithm}
In weak approximation, our main interest is the evaluation of
$$
\mathbb{E}(f(X(T))),
$$
where $f$ is Lipschitz payoff function with Lipschitz constant $D_L$.
The expectation above can be estimated with a standard Monte Carlo method that calculates a mean value of $K$ independent samples from $f(X_{M,n}^{RE})$. 
\begin{remark}
\label{mc_cost}
Since Monte Carlo estimator evaluates the payoff function only once per trajectory, the informational cost of the evaluation of $f$ can be neglected in the overall informational cost. Therefore, the total informational cost of the estimator is defined as the overall informational cost for the evalutation of every trajectory itself, namely
\begin{equation*}
\cost{\Big{(}\frac{1}{K}\sum_{i=1}^{K}f(X_{M,n}^{RE,i})\Big{)}} := \cost((X_{M,n}^{RE,i})_{i=1}^{K}) = K M n. 
\end{equation*}
\end{remark}

First, we focus on $L^2(\Omega)$-error of Monte Carlo estimator. Note that $f$ is Lipschitz payoff with Lipschitz constant $D_L$ and $\|X_{M,n}^{RE}\|_{L^{2}(\Omega)}$ is bounded by constant that depends only on parameters of class $\mathcal{F}(C,D,D_L,\Delta, \varrho_1, \varrho_2, \nu)$ (see Lemma 8 in \cite{SIAM}). Therefore, variance of $f(X_{M,n}^{RE})$ is bounded by $\kappa_{2} \geq 0$ that depends only on parameters of the aforementioned class. Rewriting the mean squared error of the standard Monte Carlo estimator as sum of its variance and squared bias indicates that the method has the following upper error bound 
\begin{equation*}
\begin{split}
\mathbb{E}\big{|} \mathbb{E}(f(X(T))) - \frac{1}{K} \sum_{i=1}^{K} f(X_{M,n}^{RE, i}) \big{|}^{2} 
& = \frac{\var(f(X_{M,n}^{RE}))}{K} + \Big{(}\mathbb{E}\big{[}f(X(T)) - f(X_{M,n}^{RE})\big{]}\Big{)}^{2} \\ 
& \leq \kappa_{2} K^{-1} + D_L^{2} (\mathbb{E} \| X(T) - X_{M,n}^{RE}\|)^{2} \\
& \leq \kappa_{2} K^{-1} + 2 (\kappa D_L)^{2}(n^{-2\alpha} + \delta^{2}(M)),
\end{split}
\end{equation*}
which results from
theorem \ref{upper_bound} for bias and the fact that the variance of $f(X_{M,n}^{RE})$ is bounded by the constant.
Therefore, one obtains that
\begin{equation}
\label{mc_eq_1}
\| \mathbb{E}(f(X(T))) - \frac{1}{K} \sum_{i=1}^{K} f(X_{M,n}^{RE, i}) \|_{L^{2}(\Omega)} 
\leq \kappa_{3}(K^{-1/2}+n^{-\alpha} + \delta(M))     
\end{equation}
where $\kappa_3:=\sqrt{\kappa_2} \vee \sqrt{2}\kappa D_L.$
Let
\begin{equation}
\label{mc_eps_def}
\mc(\varepsilon) := \frac{1}{K(\varepsilon)}\sum_{i=1}^{K(\varepsilon)}f(X_{M(\varepsilon), n(\varepsilon)}^{RE,i})
\end{equation}
where
\begin{equation*}
K(\varepsilon) := \lceil \varepsilon^{-2} \rceil,
n(\varepsilon) := \lceil \varepsilon^{-1/\alpha} \rceil,  
M(\varepsilon) := \lceil \delta^{-1}(\varepsilon) \rceil.
\end{equation*}
We next focus on informational cost of $\mc(\varepsilon)$ depending on $\varepsilon>0.$ Note that, $\lceil x \rceil \leq (1+\frac{1}{x_0})x$ for $0<x_0 \leq x,$
thus
\begin{equation*}
\begin{split}
    \varepsilon^{-2} \leq & \ K = \lceil \varepsilon^{-2} \rceil \leq (1 + (1 \wedge \delta(1))^{2})\varepsilon^{-2} \leq 2 \varepsilon^{-2}, \\
    \varepsilon^{-1/\alpha} \leq & \ n = \lceil \varepsilon^{-1/\alpha} \rceil \leq  (1 + (1\wedge \delta(1))^{1/\alpha})\varepsilon^{-2} \leq 2 \varepsilon^{-1/\alpha},
\end{split}
\end{equation*}
and
\begin{equation*}
    \delta^{-1}(\varepsilon) \leq M = \lceil \delta^{-1}(\varepsilon) \rceil \leq 2 \delta^{-1} (\varepsilon)
\end{equation*}
for sufficiently small $\varepsilon$, i.e less than or equal to $1\wedge \delta(1).$ Therefore, following remark \ref{mc_cost}, we obtain tight informational cost bounds 
\begin{equation*}
\varepsilon^{-(2+\frac{1}{\alpha})}\delta^{-1}(\varepsilon) \leq \cost{\Big{(}\mc(\varepsilon)\Big{)}} \leq 8 \varepsilon^{-(2+\frac{1}{\alpha})}\delta^{-1}(\varepsilon).
\end{equation*}
Finally, 
from \eqref{mc_eq_1}
and obtained bounds, the following fact arises
\begin{fact}
\label{fact_mc_err_cost}
For $\varepsilon>0,$ one obtains that
\begin{equation*}
\| \mathbb{E}(f(X(T))) - \mc(\varepsilon)\|_{L^{2}(\Omega)} = \bigo(\varepsilon)
\end{equation*}
and
\begin{equation*}
\cost(\mc(\varepsilon)) = \Theta(\varepsilon^{-(2+\frac{1}{\alpha})}\delta^{-1}(\varepsilon)).
\end{equation*}
\end{fact}

\section{Multilevel Monte Carlo method}\label{sec:mlmc_algorithm}
In this section, we investigate properties of MLMC estimator leveraged for the weak approximation problem introduced in the previous section.

Let $(n_l)_{l=0}^{+\infty} \subset \mathbb{N}$ and $(M_l)_{l=0}^{+\infty} \subset \mathbb{N}$ be two fixed non-decreasing sequences of parameters. For parameters $L \in \mathbb{N}$ and $(K_l)_{l=0}^{L} \subset \mathbb{N},$ MLMC estimator of $\mathbb{E}(f(X(T)))$ is defined as
\begin{equation}
\label{MC}
    \mlmc := \frac{1}{K_{0}} \sum_{i_{0}=1}^{K_0} f(X_{M_{0},n_{0}}^{RE, i_{0}}) + \sum_{l=1}^{L} \frac{1}{K_{l}} \sum_{i_{l}=1}^{K_l}\big{(} f(X_{M_{l},n_{l}}^{RE, i_{l}}) - f(X_{M_{l-1},n_{l-1}}^{RE, i_{l}})\big{)}.
\end{equation}
Note that, in this paper, MLMC estimator depends on a set of parameters both for grid densities and \textbf{truncation parameters}.
As $(n_l)_{l=0}^{+\infty}$ and $(M_l)_{l=0}^{+\infty} \subset \mathbb{N}$ remain fixed, the $L^{2}(\Omega)$-error bound of the estimator can be attained under minimal informational cost if one uses the optimal values for $L$ and $(K_l)_{l=0}^{L}$ (see \cite{GILES}). Such values are derived in the following part of this section.

From now, we assume that for each level $l=1,..., L,$ and $i_l \in \{1,..., K_l\}$ random variables $X_{M_l, n_l}^{RE, i_l}$ and $X_{M_{l-1}, n_{l-1}}^{RE, i_l}$ are coupled only via the use of the same realization of Wiener process, Poisson process, and jump-heights sequence. Drift randomizations between different levels
remain independent of each other.  

First, the very same realization of a Wiener process can be approximated under two different grid densieties $n_l$ and $n_{l-1},$ and thus be referred to as a coupling between Wiener process approximations. It can be achieved if one generates an approximated trajectory with grid denstiy of the least common multiple of $n_l$ and $n_{l-1},$ and then sum its increments between given gridpoints (see section \ref{sec:impl} for more implementation details). 

Next, jump times of a given Poisson process are generated regardless of a grid density by sampling from exponential distribution and then assiging their locations to the right bounds of grid intervals that they fall into accordingly. 

Finally, common jump heights sequence is generated for a given number of jump times using its common distribution.

In the next part of this section we derive optimal parameters for $\mlmc$ estimator and proove that the variance of subsequent levels vanishes to $0$ which is a necessary condition for the algorithm to converge. We also provide the reasoning behind why drift randomizations do not need to coupled. 

First of all, one may define the timestep of the algorithm as $h_l:=T/n_l$ as in \eqref{main_scheme}. Nevertheless, it is convenient to have timestep $h_l$ in the form of
\begin{equation*}
h_l := T \beta^{-l},
\end{equation*} for some $\beta > 1.$ Therefore, we usually assume that the grid densits are defined as 
\begin{equation*}
n_l:=\lceil \beta^{l}\rceil
\end{equation*} for $\beta>1$ and every $l=0,\dots,L,$ so that 
\begin{equation}
\label{nl_hl_2}
\frac{T}{n_l} = \frac{T}{\lceil \beta^l \rceil} \leq T \beta^{-l} = h_l.
\end{equation}
It means that our timesteps do not exceed the desired $h_l$. Furthermore, note that 
\begin{equation*}
n_l^{-\alpha} \leq T^{-\alpha} h_l^{\alpha}
\end{equation*}
for $\alpha > 0,$ which means that we can rewrite thesis of the theorem \ref{upper_bound} in terms of $h_l$ instead of $n_l$ with new constant $\tilde{\kappa}:= ( T^{-\alpha}\vee 1)\kappa.$
Since $\beta > 1,$ we have that $\lceil \beta^{l} \rceil \leq 2 \beta^{l}.$ 
Thus,
\begin{equation}
\label{nl_hl_1}
n_l = \lceil \beta^l \rceil \leq 2 \beta^{l} = 2T \frac{\beta^l}{T} = \frac{2T}{h_l}.
\end{equation}

As already mentioned, mean squared error can be rewritten as sum of variance and squared bias, which can also be applied to $L^{2}(\Omega)$-error multilevel estimator. This leads to the following equality
$$
\| \mathbb{E}(f(X(T))) - \mlmc \|_{L^{2}(\Omega)}^2 = \var(\mlmc) + (\mathbb{E}(f(X(T))) -\mathbb{E}(\mlmc))^2.
$$
Note that $\mathbb{E}(\mlmc) = \mathbb{E}(f(X_{M_L,n_L}^{RE})),$
thus
\begin{equation}
\label{bias_bound}
\begin{split}
(\mathbb{E}(f(X(T))) - \mathbb{E}(\mlmc))^2 
\leq \mathbb{E}\big{|} f(X(T)) - f(X_{M_{L}, n_{L}}^{RE}) \big{|}^{2}  
\leq \kappa_{bias} (h_{L}^{2\alpha}
+ \delta^{2}(M_{L})),
\end{split}
\end{equation}
where
$\kappa_{bias}:=2(D_L \tilde{\kappa})^{2}.$
Therefore, the upper bound for the second term above (squared bias) depends only on the convergence rate of the numerical scheme which is determined by $\alpha$ and $\delta(\cdot)$. 

The MLMC method aims at variance reduction of the estimator to reduce the informational cost. 
In the following part, we provide optimal values for the remaining parameters of MLMC estimator. To investigate the cost of the MLMC method for SDEs driven by the countably dimensional Wiener process, we replicate steps presented in paper \cite{GILES}. 
Let
\begin{equation*}
v_{l} := \var \big{[} f(X_{M_{l}, n_{l}}^{RE}) - f(X_{M_{l-1}, n_{l-1}}^{RE})\big{]},
\end{equation*}
and
\begin{equation*}
v_{0} := \var \big{[} f(X_{M_{0}, n_{0}}^{RE})\big{]}.
\end{equation*}
One can notice that variance of MLMC estimator is equal to
\begin{equation*}
\var[\mlmc] = \sum_{l=0}^{L}\frac{v_{l}}{K_{l}}
\end{equation*}
and, following the remark \ref{mc_cost}, the cost can be defined as
\begin{equation*}
\cost(\mlmc) := \sum_{l=0}^{L} K_{l} M_{l} n_l
\end{equation*}
or alternatively (by inequalities \eqref{nl_hl_2} and \eqref{nl_hl_1}) as
$
\sum_{l=0}^{L} \frac{K_{l} M_{l}}{h_{l}}.$
By minimizing the variance with respect to the fixed cost, 
one obtains that
the optimal value for $K_{l}$ is 
\begin{equation}
\label{optimal_kl}
    K_l = \Big{\lceil} 2 \varepsilon^{-2} \sqrt{\frac{v_l h_l}{M_l}} \sum_{k=0}^{L}\sqrt{\frac{v_k M_k}{h_k}} \Big{\rceil}
\end{equation}
where $\varepsilon$ is the expected $L^{2}(\Omega)$ error of the algorithm. 
Recall that for any two square-integrable random variables $X, Y$ one has that
\begin{equation*}
    \var(X-Y) 
    \leq 
    (\var^{\frac{1}{2}}(X)+\var^{\frac{1}{2}}(Y))^{2}.
\end{equation*}
On the other hand, since
\begin{equation*}
h_{l-1} = \beta h_l \ \text{and} \ \delta(M_{l}) < \delta(M_{l-1}) 
\end{equation*}
for $l=1,\dots, L,$ we have that
\begin{equation*}
\begin{split}
v_l & = \var\big{[}\big{(}f(X(T))-f(X_{M_{l-1}, n_{l-1}}^{RE})\big{)} - \big{(}f(X(T))-f(X_{M_{l}, n_{l}}^{RE})\big{)}\big{]}
\\ & \leq \big{[}\big{(}\var(f(X(T)) - f(X_{M_{l}, n_{l}}^{RE}))\big{)}^{\frac{1}{2}} + \big{(} \var(f(X(T)) - f(X_{M_{l-1}, n_{l-1}}^{RE}))\big{)}^{\frac{1}{2}}\big{]}^{2}
\\ 
& \leq 2D_L^{2}\|X(T) - X_{M_l, n_l}^{RE} \|_{L^{2}(\Omega)}^{2} + 2D_L^{2}\| X(T) - X_{M_{l-1}, n_{l-1}}^{RE}\|_{L^{2}(\Omega)}^{2} \\ 
& \leq 8(1 \vee \beta^{2\alpha})(D_L \tilde{\kappa})^2(h_l^{2\alpha} + \delta^{2}(M_{l-1}))
\end{split}.
\end{equation*}
Similarly, from the fact that variance of $f(X_{M_0, n_0}^{RE})$ is bounded 
and the observation that $\delta$ takes only positive values, we obtain that
\begin{equation*}
\begin{split}
v_0  \leq \mathbb{E}\big{|}f(X_{M_0,n_0}^{RE})\big{|}^{2} \leq \kappa_2 
& = \frac{\kappa_2}{T^{2\alpha} + \delta^{2}(M_0)}(T^{2\alpha} + \delta^{2}(M_0)) \\
& \leq 
\frac{\kappa_2}{T^{2 \alpha}}(h_{0}^{2\alpha} + \delta^{2}(M_{0}))
\end{split}
\end{equation*}
Hence,
\begin{equation}
\label{variance_bound}
v_l \leq 
\kappa_{var}(h_l^{2\alpha} + \delta^{2}(M_{l-1})), \ l \in \{0,\dots, L\} 
\end{equation}
where $\kappa_{var}:=4\beta^{2\alpha}\kappa_{bias} \vee \kappa_{2}T^{-2\alpha}$ and $M_{-1}:=M_{0}.$ 

From inequality \eqref{variance_bound} we obtain that $v_l \to 0$ as $l \to +\infty$ which guarantees that one needs fewer and fewer samples on the next levels. Similarly to SDEs driven by the finite-dimensional Wiener process, it means that coarse levels contribute to the cost with a slightly greater number of independent samples. Inequality \eqref{variance_bound} also indicates that despite the fact that for $l=1,...,L,$ sequences of independent random variables $\theta_{j}^{(l)} \sim \mathcal{U}(j\frac{T}{n_l}, (j+1)\frac{T}{n_l})$ for $j=0,...,n_l-1$ and 
$\theta_{j}^{(l-1)} \sim \mathcal{U}(j\frac{T}{n_{l-1}}, (j+1)\frac{T}{n_{l-1}})$ for $j=0,...,n_{l-1}-1$  are not coupled, both $f(X_{M_{l-1},n_{l-1}}^{RE})$ and $f(X_{M_{l},n_{l}}^{RE})$ approximate the same realization of a random variable $f(X(T))$ which is sufficient for the algorithm to work. 

The general idea for the proof of complexity bounds for the MLMC method consists of a few repeatable steps. To establish the parameters of an algorithm, first, we try to find values for which the upper error bound $\varepsilon>0$ is attained. The parameters depend on $\varepsilon.$ It usually starts with parameter $L$ since it determines the upper bound for the bias of an estimator which is entirely determined by the convergence rate of the numerical scheme and not by the method itself (see inequality \eqref{bias_bound}). Next, we proceed with the number of summands $K_l$ which affects the estimator's variance.
Finally, having obtained concrete parameters' values for which the upper error bound is attained, we check the corresponding complexity of an algorithm in terms of $\varepsilon$ (upper bound for the $L^2(\Omega)$ error).

In the next part of this paper, we provide a theorem that addresses the complexity upper bound for the multilevel algorithm in setting \eqref{main_equation}. As expected, the cost highly depends on $\delta$ sequence. In the general case, the dependence relates to the exponent of a possibly unknown constant.  
Nevertheless, the impact of the exponent can be mitigated in the following family of classes for the inverse of $\delta.$

\begin{definition}
We define the following family of classes
\begin{equation*}
\begin{split}
    &\mathcal{G}_{x_0} := \Big{\{} 
    g: (0,x_0) \mapsto \mathbb{R}_{+}: g-\mbox{strictly decreasing, and} \\
    &\quad\quad\quad \exists_{\tilde{C}: \mathbb{R}_{+} \mapsto \mathbb{R}_{+}} \ \forall_{0<x,y < x_0}: 
    \frac{g(y)}{g(x)}\leq \tilde{C}\Big{(}\frac{\log(x_0/y)}{\log(x_0/x)}\Big{)} \Big{\}}
\end{split}
\end{equation*}
which is parametrized by $x_0 >0.$ The class $\mathcal{G}_{x_0}$ is further called the class of positive log-decreasing functions.
For the sake of brevity, if function $g$ is defined on a domain broader than $(0, x_0),$ by condition $g \in \mathcal{G}_{x_0}$ we usually mean $g|_{(0,x_0)} \in \mathcal{G}_{x_0}.$
\end{definition}

\begin{fact}
For any $x_0 > 0,$ class $\mathcal{G}_{x_0}$ satisfies the following properties:
\begin{enumerate}
    \item [(P1)] $\mathcal{G}_{x_0}$ is non-empty since
    \begin{equation*}
g: (0,x_0) \ni x \mapsto \log(x_0/x) \in \mathbb{R}_{+}
\end{equation*} 
belongs to the class with $\tilde{C}(x) = x.$
    \item [(P2)] For every $g_1, g_2 \in \mathcal{G}_{x_0},$ one has that
    \begin{equation*}
    g_1 + g_2: (0, x_0) \ni x \mapsto g_{1}(x) + g_{2}(x) \in \mathbb{R}_{+}
    \end{equation*}
    and
    \begin{equation*}
    g_1 g_2: (0, x_0) \ni x \mapsto g_1(x) g_2(x) \in \mathbb{R}_{+}
    \end{equation*}
    belong to $\mathcal{G}_{x_0}.$
    \item [(P3)] For every $g \in \mathcal{G}_{x_0}$ and $a > 0,$ one has that
    \begin{equation*}
    a + g: (0, x_0) \ni x \mapsto a + g(x) \in \mathbb{R}_{+}
    \end{equation*}
    and
    \begin{equation*}
    a\cdot g: (0, x_0) \ni x \mapsto a \cdot g(x) \in \mathbb{R}_{+}
    \end{equation*}
    belong to $\mathcal{G}_{x_0}.$
    \item [(P4)] For every $g \in \mathcal{G}_{x_0}$ and $\alpha > 0,$ one has that
    \begin{equation*}
    g^{\alpha}: (0, x_0) \ni x \mapsto (g(x))^{\alpha} \in \mathbb{R}_{+}
    \end{equation*}
     belongs to $\mathcal{G}_{x_0}.$
     \item [(P5)] \label{class_g_facts}
     For any $g \in \mathcal{G}_{x_0}$, $a>0$ and $x \in (0,x_0^{\frac{1}{a}})$ one obtains that
    \begin{equation*}
    g(x^a) \leq \tilde{C}(a)g(x_0^{\frac{a-1}{a}}x), 
    \end{equation*}
    which results from the direct substitution $y:=x_0^{1-a}x^{a}.$
\end{enumerate}
\end{fact}
Properties (P1)-(P4) guarantee that for any $x_0 > 0$ class $\mathcal{G}_{x_0}$ is rich in various functions. In fact, for any $x_0>0$ class $\mathcal{G}_{x_0}$ is a convex cone. On the other hand, property (P5) assures that the exponent can be reduced to the multiplicative constant.

The following theorem provides the upper bound for the cost of the multilevel algorithm for the considered class of SDEs. We also provide an additional lower cost bound of the estimator to stress that the cost is greater than the one measured in a finite-dimensional setting. In compliance with \cite{GILES}, the cost in finite-dimensional setting is proportional to $\varepsilon^{-2}(\log(\varepsilon))^{2}$.

\begin{theorem}
\label{main_thm}
Let $(a,b,c,\eta) \in \mathcal{F}(C,D,D_L,\Delta, \varrho_1, \varrho_2, \nu)$ be the tuple of functions defining equation \eqref{main_equation} with  $\varrho_1, \varrho_2 \in [1/2,1],$ so that  $\alpha=1/2.$ For any sufficiently small $\varepsilon \geq 0$
 there exists a multilevel algorithm $\mlmc$ such that
 \begin{enumerate}
\item [i)] 
$$
\|\mlmc - \mathbb{E}(f(X(T)))\|_{L^{2}(\Omega)} \leq \varepsilon,
$$
\item [ii)]
\begin{equation*}
\cost({\mlmc}) \leq 
\begin{cases}
c_7 \varepsilon^{-2} (\log(\varepsilon^{-1}))^{2} \delta^{-1}(\varepsilon), \ \mbox{for} \ \delta^{-1} \in \mathcal{G}_{\delta(1)} \\
c_6 \varepsilon^{-2} (\log(\varepsilon^{-1}))^{2} \delta^{-1}(\varepsilon^{\kappa_{cost}}), \ \mbox{otherwise}
\end{cases},
\end{equation*}
\item [iii)]
\begin{equation*}
    \cost({\mlmc}) \geq c_9 \varepsilon^{-2}
    \Big{(}(\log(\varepsilon^{-1}))^{2} + \delta^{-1}(\varepsilon)\Big{)}.
\end{equation*}
 \end{enumerate}
for some positive constants $c_6, c_7, c_9$ and $\kappa_{cost}>1,$ depending only on the parameters of the class $\mathcal{F}(C,D,D_L,\Delta, \varrho_1, \varrho_2, \nu)$ and some $\beta > 1.$
\end{theorem}

\begin{proof}
Without loss of generality, let $D_L>1.$ Otherwise, note that any $D_L$-Lipschitz function satisfies Lipschitz condition with constant  $D_L \vee 1$. Similarly w.l.o.g 
we assume that $\delta(1) > 1.$ 

Next, let 
\begin{equation*}
n_l := \lceil \beta^{l} \rceil
\end{equation*} for some $\beta > 1,$
so that
\begin{equation*}
h_l := T \beta^{-l} \geq \frac{T}{n_l}.
\end{equation*}
Knowing the exact rate of convergence which depends on $\alpha$ and $\delta,$ let 
\begin{equation*}
M_l :=\lceil \delta^{-1}(\beta^{-\alpha(l+1)})\rceil,
\end{equation*}
thus
\begin{equation*}
\begin{split}
& M_{l} \geq \delta^{-1}(\beta^{-\alpha (l+1)}) \\
& \delta(M_{l}) \leq \beta^{-\alpha (l+1)} = \Big{(}\frac{1}{T \beta}\Big{)}^{\alpha}h_l^{\alpha}.
\end{split}
\end{equation*}
Note that $\beta^{-\alpha(l+1)}$ falls into the domain of $\delta^{-1}$ which is $(0, \delta(1)],$ since $\beta > 1 \geq \delta(1)^{-1/\alpha}.$
From inequality \eqref{bias_bound} one obtains the following upper bound 
\begin{equation*}
\begin{split}
& (\mathbb{E}(f(X(T))) - \mathbb{E}(\mlmc))^2 \leq \kappa_{bias}(h_{L}^{2\alpha}
+ \delta^{2}(M_{L}))\leq 2(1 \vee (T \beta)^{-2\alpha})\kappa_{bias} h_L^{2\alpha} = c_1 h_L^{2\alpha}
\end{split}
\end{equation*}
where $c_1 := 2(1 \vee (T \beta)^{-2\alpha})\kappa_{bias}.$
Having set 
\begin{equation}
\label{l_value_proof}
L := \Big{\lceil} \frac{\log{(\sqrt{2c_1}T^{\alpha} \varepsilon^{-1}})}{\alpha\log{\beta}} \Big{\rceil},
\end{equation}
we get the desired upper bound for squared bias
\begin{equation}
\label{sqr_bias_proof}
(\mathbb{E}(f(X(T))) - \mathbb{E}(\mlmc))^{2} \leq \frac{\varepsilon^{2}}{2}.
\end{equation}
Similarly, for such parameters $n_l$ and $M_l$ and from inequality  \eqref{variance_bound} one obtains that
\begin{equation*}
\begin{split}
v_l \leq \kappa_{var}(h_l^{2\alpha} + \delta^{2}(M_{l-1})) 
\leq 2(1 \vee T^{-2\alpha})\kappa_{var} h_l^{2\alpha} = c_2 h_l^{2\alpha}
\end{split}
\end{equation*}
where
$$
c_2 := 2(1 \vee T^{-2\alpha})\kappa_{var}.
$$
Recalling that $\alpha=1/2,$ the above inequality simplifies to
\begin{equation*}
v_l \leq c_2 h_l.
\end{equation*}
Utilizing the above property, we get the following upper bound for the optimal $K_l.$ Following the idea presented in \cite{GILES}, the actual value for $K_l$ in the proof is further chosen to be equal to the obtained upper bound, namely 
\begin{equation*}
\begin{split}
    K_l & :=
    \Big{\lceil} 2 c_2\varepsilon^{-2} \frac{h_l}{\sqrt{M_l}} \sum_{k=0}^{L}\sqrt{M_k}
    \Big{\rceil}
    =
     \Big{\lceil} 2 c_2\varepsilon^{-2} \frac{h_l^{\alpha+1/2}}{\sqrt{M_l}} \sum_{k=0}^{L}h_k^{\alpha-1/2}\sqrt{M_k}
    \Big{\rceil} \\
    & \geq
    \Big{\lceil} 2 \varepsilon^{-2} \sqrt{\frac{v_l h_l}{M_l}} \sum_{k=0}^{L}\sqrt{\frac{v_k M_{k}}{h_k}} 
    \Big{\rceil}.
\end{split}
\end{equation*}
Note, that
\begin{equation}
    \label{ml_bound}
    1 \leq \frac{\sum_{k=0}^{L}\sqrt{M_k}}{\sqrt{M_l}} 
\end{equation}
for any $L \in \mathbb{N}$ and $l=0,\dots,L.$
Thus, from the above inequality in \eqref{ml_bound}, we have that
\begin{equation}
\label{var_proof}
\var(\mlmc) =\sum_{l=0}^{L}\frac{v_l}{K_l} \leq \sum_{l=0}^{L}\frac{c_2 h_l}{2 c_2 \varepsilon^{-2}h_l} = \frac{\varepsilon^{2}}{2}.
\end{equation}
Henceforth, together from \eqref{sqr_bias_proof} and \eqref{var_proof}, we obtain that
\begin{equation*}
\| \mathbb{E}(f(X(T))) - \mlmc \|_{L^{2}(\Omega)} \leq \varepsilon,
\end{equation*}
which means that the mean square error of our algorithm does not exceed the desired upper bound. Next, we focus on the value of $\cost(\mlmc).$

From assumption that $\beta>1$ we obtain
\begin{equation*}
M_l = \lceil \delta^{-1}(\beta^{-\alpha(l+1)}) \rceil \leq \Big{(}1+\frac{1}{\delta^{-1}(1)}\Big{)} \delta^{-1}(\beta^{-\alpha(l+1)})    
\end{equation*}
for every $l=0,\dots,L.$
Assuming that $\varepsilon$ is sufficiently small, i.e $\varepsilon < 1/\euler,$ from \eqref{l_value_proof} one obtains that
$$
L+1 \leq c_4 \log(\varepsilon^{-1})
$$
for
$$
c_4 := \Big{(}0 \vee \frac{\log(\sqrt{2c_1}T^{\alpha})}{\alpha \log \beta} \Big{)} + \frac{1}{\alpha \log \beta} + 2
$$
which was also stressed in \cite{GILES}.
From this property, we obtain that
\begin{equation*}
\begin{split}
\delta^{-1}(\beta^{-\alpha(L+1)}) & \leq 
\delta^{-1}(\varepsilon^{c_4 \alpha \log \beta})
= \delta^{-1}(\varepsilon^{\kappa_{cost}}),
\end{split}
\end{equation*}
where $\kappa_{cost}:=c_4\alpha \log \beta.$
Therefore,
\begin{equation*}
M_l \leq c_5 \delta^{-1}(\varepsilon^{\kappa_{cost}})
\end{equation*}
for $l=0,\dots,L$ where $c_5:=\Big{(}1+\frac{1}{\delta^{-1}(1)} \Big{)}.$
Since 
\begin{equation*}
K_l \leq 2c_2 \varepsilon^{-2}\frac{h_l}{\sqrt{M_l}}\sum_{k=0}^{L}\sqrt{M_k} +1,
\end{equation*}
the following inequality holds true
\begin{equation*}
\begin{split}
K_{l}M_{l}n_{l} 
\leq (2c_2 \varepsilon^{-2}\frac{h_l}{\sqrt{M_l}}\sum_{k=0}^{L}\sqrt{M_k} + 1)M_l\frac{2T}{h_l}
\end{split}.
\end{equation*}
Thus, it results in
\begin{equation*}
\begin{split}
\sum_{l=0}^{L}K_{l}M_{l}n_{l}  & 
\leq \sum_{l=0}^{L} (2c_2 \varepsilon^{-2}\frac{h_l}{\sqrt{M_l}}\sum_{k=0}^{L}\sqrt{M_k} + 1)M_l\frac{2T}{h_l} \\
& \leq 2T \Big{(}2c_2 \varepsilon^{-2} +\sum_{l=0}^{L} h_l^{-1}\Big{)}(L+1)^{2}M_{L}.
\end{split}
\end{equation*}
Similarly as in \cite{GILES}, note that 
\begin{equation*}
h_l^{-1}=T^{-1}\beta^{l} = T^{-1} \beta^{L} \beta^{l-L} = h_L^{-1} \beta^{l-L}
\end{equation*}
and from \eqref{l_value_proof} 
\begin{equation*}
    \begin{split}
        & L - 1 \leq \frac{\log(\sqrt{2 c_1}T^{\alpha}\varepsilon^{-1})}{\alpha \log \beta} \\
        & h_L^{-\alpha} = (T^{-1} \beta^{L})^{\alpha} \leq \sqrt{2 c_1} \beta^{\alpha} \varepsilon^{-1} \\
        & h_L^{-1} \leq (\sqrt{2 c_1})^{\frac{1}{\alpha}} \beta \varepsilon ^{-1/\alpha} = 2c_1 \beta \varepsilon^{-2}.
    \end{split}
\end{equation*}
From these inequalities, one obtains that
\begin{equation*}
\sum_{l=0}^{L}h_l^{-1} = h_L^{-1}\sum_{l=0}^{L}\beta^{l-L} 
\leq \frac{(\sqrt{2 c_1})^{\frac{1}{\alpha}}\beta^2}{\beta - 1} \varepsilon^{-2}. 
\end{equation*}
From previously obtained upper bounds we get that
\begin{equation*}
\begin{split}
\sum_{l=0}^{L}K_lM_ln_l & \leq 2T\Big{(}2c_2 + \frac{(\sqrt{2 c_1})^{\frac{1}{\alpha}}\beta^2}{\beta - 1} \Big{)}\varepsilon^{-2}(L+1)^{2}M_{L} \\
& \leq 2T\Big{(}2c_2 + \frac{(\sqrt{2 c_1})^{\frac{1}{\alpha}}\beta^2}{\beta - 1} \Big{)}c_{4}^{2}c_{5}\varepsilon^{-2}(\log(\varepsilon^{-1}))^{2}\delta^{-1}(\varepsilon^{\kappa}) \\
& = c_6 \varepsilon^{-2}(\log(\varepsilon^{-1}))^{2}\delta^{-1}(\varepsilon^{\kappa})
\end{split}
\end{equation*}
where $c_6:=2T\Big{(}2c_2 + \frac{(\sqrt{2 c_1})^{\frac{1}{\alpha}}\beta^2}{\beta - 1} \Big{)}c_{4}^{2}c_{5}.$ 
Note that $\kappa_{cost}=c_4 \alpha \log \beta>1,$ thus 
\begin{equation*}
\delta^{-1}(\varepsilon^{\kappa_{cost}}) \geq \delta^{-1}(\varepsilon).
\end{equation*} Furthermore, if $\varepsilon<(\delta(1))^{1/\kappa_{cost}}$ and $\delta^{-1} \in \mathcal{G}_{\delta(1)},$ from property (P5) in fact \ref{class_g_facts}, we obtain that
\begin{equation*}
    \delta^{-1}(\varepsilon^{\kappa_{cost}}) \leq \tilde{C}(\kappa_{cost}) \delta^{-1}(c_8 \varepsilon) \leq \tilde{C}(\kappa_{cost}) \delta^{-1}(\varepsilon)
\end{equation*}
where $c_8:= (\delta(1))^{\frac{\kappa_{cost}-1}{\kappa_{cost}}} > 1.$
It completes the proof for the upper complexity bounds of the algorithm with $c_7:=c_6 \tilde{C}(\kappa_{cost})$. 

We now proceed with additional lower complexity bound. 
From 
\begin{equation*}
 \kappa_{bias} \delta^{2}(M_L) \leq \kappa_{bias}(h_{L}^{2\alpha}
+ \delta^{2}(M_{L})) \leq \varepsilon^{2}
\end{equation*} and the observation that $\sqrt{\kappa_{bias}} = \sqrt{2} D_L \tilde{\kappa} > 1,$
we obtain that
\begin{equation*}
    M_{L} \geq \delta^{-1}\Big{(}\frac{\varepsilon}{\sqrt{\kappa_{bias}}}\Big{)} > \delta^{-1}(\varepsilon).
\end{equation*}
Similarly, from the lower bounds on $K_l$ and $n_{l}$ we obtain
\begin{equation*}
\begin{split}
\cost({\mlmc}) & 
\geq \sum_{l=0}^{L}\Big{(}2c_2 \varepsilon^{-2}\frac{h_l}{\sqrt{M_l}}\sum_{k=0}^{L}\sqrt{M_k}\Big{)}\frac{2T M_l}{h_l} \\
& \geq 4T c_2 \varepsilon^{-2}\Big{(} L^{2} + M_{L}\Big{)} \\
& \geq 4T c_2 (1 \wedge (\alpha \log \beta)^{-1}) 
\varepsilon^{-2}\Big{(} \log(\sqrt{2 c_1}T^{\alpha} \varepsilon^{-1}) + \delta^{-1}(\varepsilon)\Big{)} \\
& \geq c_9 \varepsilon^{-2}\Big{(}(\log(\varepsilon^{-1}))^{2} + \delta^{-1}(\varepsilon)\Big{)}
\end{split}
\end{equation*}
where 
\begin{equation*}
c_9 := 4T c_2 (1 \wedge (\alpha \log \beta)^{-1})
\end{equation*}
and 
\begin{equation*}
\begin{split}
& \sqrt{2c_1}T^{\alpha} = 2 \sqrt{\kappa_{bias}}( T^{\alpha} \vee \beta^{-\alpha}) \\
& \geq 2 \sqrt{\kappa_{bias}} T^{\alpha} = 2\sqrt{2} D_L \tilde{\kappa} T^{\alpha} \\
& = 2\sqrt{2} D_L  (1 \vee T^{-\alpha}) T^{\alpha} \kappa \geq 2\sqrt{2} D_L \kappa > 1,
\end{split}
\end{equation*}
which completes the proof.
\end{proof}

\section{Numerical experiments}
\label{sec:numer}
In this section, we compare results from numerical experiments carried out for standard Monte Carlo method defined in section \ref{sec:mc_algorithm} and MLMC method defined in section \ref{sec:mlmc_algorithm}. 

Some of the parameters in the definition of MLMC remain unknown, since they depend on the variance of corresponding levels (see equation \eqref{optimal_kl}). In subsection \ref{sec:mlmc_params} we overcome this issue by following the approach presented in \cite{GILES}, and hence introduce the approximate value of MLMC estimator with $L^{2}(\Omega)$-error bound of $\varepsilon>0,$ which is further denoted by $\widehat{\mlmc}(\varepsilon).$

In subsection section \ref{sec:example} we introduce example equation and provide results from numerical experiments perfmored for it.
For $\varepsilon>0,$ we check estimated $L^{2}(\Omega)$-error of $\mc(\varepsilon)$ (see equation \eqref{mc_eps_def} for the definition) and relation with its informational cost intorduced in remark \ref{mc_cost}. The log-log plot for error vs cost is compared with the actual theoretical slope according to fact \ref{fact_mc_err_cost}. Similarily, we compare estimated $L^{2}(\Omega)$ error of $\widehat{\mlmc}(\varepsilon)$ with its estimated informational cost and plot the theoretical upper bound for the cost (see Theorem \ref{main_thm}). The unknown constants in the upper bound for the cost are estimated with \texttt{curve\_fit} procedure from \texttt{scipy.optimize} subpackage.

Finally, for the convenience of the reader, in subsection \ref{sec:impl} we provide implementation details.

\subsection{Multilevel Monte Carlo parameters}
\label{sec:mlmc_params}
In numerical experiments, the main parameters of the multilevel method are set as defined in the proof of theorem \ref{sec:mc_algorithm} with $\beta=2,$ i.e $n_{l}=2^{l}$ and $M_{l}=\lceil \delta^{-1}(2^{-(l+1)/2})\rceil$ for $l \in \mathbb{N}.$
The following part of this subsection refers to the remaining parameters of an algorithm which are dynamically estimated with the procedure presented in \cite{GILES}. 

Until the next subsection let superscript of any estimator denote the iteration number of an algorithm. For example, let $\widehat{L}^{(i)}$ denote the number of levels (excluding zero-level) at $i$-th iteration of the procedure. The number of levels is updated according to the following formula
\begin{equation*}
\widehat{L}^{(i+1)} = 
\begin{cases}
    \widehat{L}^{(i)}, \ (i>2) \ \mbox{and} \  ({convergence\_error}(i) < 0) \\
    \widehat{L}^{(i)} + 1, \ \mbox{otherwise}
\end{cases}
\end{equation*}
for $i \in \mathbb{N}$ and $\widehat{L}^{(1)}=0,$ meaning that we start our procedure with single level.
On the other hand, the final iteration id is denoted by $fin:= \min \{i \in \mathbb{N}: \widehat{L}^{(i)}=\widehat{L}^{(i+1)}\}.$
Since the optimal values for $K_l$ depend on the variance of respective levels, at $i$-th iteration one estimates $K_l$ with $\widehat{K}_{l}^{(i+1)}$ utilizing proper variance estimate $\widehat{v}_{l}^{(i)}$ of $v_l.$
The variance estimates are updated regarding the following formula
\begin{equation*}
\widehat{v}_l^{(i)} = \begin{cases}
\mbox{estimate of} \ v_l \ \mbox{with} \ 10^3 \ {samples}, \ l = \widehat{L}^{(i)} \\
\mbox{estimate of} \ v_l \ \mbox{with} \ \widehat{K}_{l}^{(i)} \ \mbox{samples}, \ l = 0,...,\widehat{L}^{(i)}-1
\end{cases}
\end{equation*}
for $l=0,..., \widehat{L}^{(i)}.$
It means that the variance of the current top level is always estimated with 1000 samples which is inspired by \cite{GILES}.
And the number of samples $\widehat{K}_{l}^{(i+1)}$ per level is always updated with respect to the following formula 
\begin{equation*}
\widehat{K}_l^{(i+1)} := \Big{\lceil} 2 \varepsilon^{-2} \sqrt{\frac{\widehat{v_l}^{(i)}}{M_l n_l}} \sum_{k=0}^{\widehat{L}^{(i)}}\sqrt{\widehat{v_k}^{(i)} M_k n_k} \Big{\rceil}
\end{equation*}
for $l=0,...,\widehat{L}^{(i)}.$
Finally, let
\begin{equation*}
\widehat{Y}_{l}^{(i)} := \frac{1}{\widehat{K}_{l}^{(i)}}\sum_{i_l=1}^{\widehat{K}_{l}^{(i)}}(f(X_{M_{l}, n_{l}}^{RE, i_{l}})-f(X_{M_{l-1}, n_{l-1}}^{RE, i_{l}})),
\end{equation*}
so that the convergence error function is defined as 
\begin{equation*}
convergence\_error: \mathbb{N} \setminus \{1, 2\} \ni i \mapsto (\frac{1}{2} \big{|}\widehat{Y}_{\widehat{L}^{(i)}-1}^{(i)}\big{|} \vee  \big{|}\widehat{Y}_{\widehat{L}^{(i)}}^{(i)}\big{|}) - (\sqrt{2}-1)\frac{\varepsilon}{\sqrt{2}}.
\end{equation*}
From now on, by $\widehat{\mlmc}(\varepsilon)$ we define the value of multilevel algorithm that uses the aforementioned estimates $\widehat{L}^{(fin)}, (\widehat{K}_{l}^{(fin)})_{l=0}^{\widehat{L}^{(fin)}}$ and defined parameters $(M_l)_{l \in \mathbb{N}}, (n_l)_{l \in \mathbb{N}}$.

\subsection{Example equation}
\label{sec:example}
In numerical experiments, we used the following equation and payoff function.

\begin{example}[\textbf{Merton model with Call option payoff}]

Let us consider the following equation
\begin{equation}
\label{sim:infBS}
    X(t) = \eta + \int\limits_0^t \mu X(s)\text{d}s + \sum_{j=1}^{+\infty}\int\limits_{0}^t \frac{\sigma_j}{j^\alpha}X(s)\text{d}W_j (s)
     + \int\limits_0^t X(s-)\text{d}L(s), \quad t\in [0,T],
\end{equation}
where $\eta, \mu \in \mathbb{R},$ $\alpha \geq 1, (\sigma_j)_{j=1}^{+\infty}$ is a bounded sequence of positive real numbers, and $L=(L(t))_{t\in[0,T]}$ is a compound Poisson process with intensity $\lambda >0$ and jump heights $(\xi_{i})_{i=1}^{+\infty}.$ The closed-form solution of equation \eqref{sim:infBS} has the following formula
\begin{eqnarray*}
        X(t) = \eta\exp\biggr[\biggr(\mu -\frac{1}{2} \sum_{j=1}^{+\infty}\frac{\sigma_j^2}{j^{2\alpha}}\biggr)t + \sum_{j=1}^{+\infty}\frac{\sigma_j}{j^\alpha}W_j(t)\biggr]\prod_{i=1}^{N(t)}(1+\xi_{i}).
\end{eqnarray*}
Note that it can be simulated on computer by truncating an infinite sums in the above formula. The truncated solution we can simulate is further denoted by $X_{M}(t)$ for $M \in \mathbb{N}.$ For $j=1,\dots,M$, independent random variables $W_{j}(t)$ are sampled from normal distribution with zero-mean and variance equal to $t.$ Accordingly, $N(t)$ is sampled from Poisson distribution with intensity $\lambda t.$ Finally, for $i=1,\dots,N(t)$ independent random variables $\xi_i$ are sampled from their common distribution. For equation \eqref{sim:infBS}, we obtain that $\delta(M) = \Theta(M^{-\alpha + 1/2}$ (see Fact 1 in \cite{SIAM}).

Next, for simulation purposes, we set 
$\mu = 0.08, \sigma_j = \sigma = 0.4$ for $j \in \mathbb{N}$ and $ \ \alpha = T = \eta = \lambda = 1$ with call-option payoff $f(x):= (x - 1 \vee 0).$ Let $(Y_{i})_{i=1}^{+\infty}$ be a sequence o  independent random variables that are normally distributed with zero mean and unit variance. We assume that the jump heights sequence of random variables is defined as
$\xi_{i} = -0.5\one_{(-\infty, 0]}(Y_{i}) + (0.5+Y_{i})\one_{(0, +\infty)}(Y_{i}).$

Since the closed-form solution of the equation is known, the value of $\mathbb{E}(f(X(T)))$ can be estimated with standard Monte Carlo method, i.e
\begin{equation*}
\mathbb{E}(f(X(T))) \approx \frac{1}{10^6}\sum_{k=1}^{10^6}f(X_{12\cdot10^3}^{(k)}(T)).
\end{equation*}
Hence, for both standard Monte Carlo and MLMC algorithms, we can estimate their respective $L^2(\Omega)$-errors with the following formula 
\begin{equation*}
    \widehat{e}_{K}(Y) := \left(\frac{1}{K}\sum_{i=1}^{K}\big|Y^{(i)}(\varepsilon) - \frac{1}{10^6}\sum_{k=1}^{10^6}f(X_{12\cdot10^3}^{(k)}(T))\big|^2\right)^{1/2},
\end{equation*}
for $Y \in \{\mathcal{MC}(\varepsilon), \widehat{\mlmc}(\varepsilon)\}.$ 
Since the cost of $\widehat{\mlmc}(\varepsilon)$ is a random variable, we estimate it with the mean cost, i.e
\begin{equation*}
    \cost(\widehat{\mlmc}(\varepsilon)) := \frac{1}{K}\sum_{i=1}^{K}\cost(\widehat{\mlmc}^{(i)}(\varepsilon)) = \frac{1}{K}\sum_{i=1}^{K}\sum_{l=0}^{\widehat{L}^{(fin), i}}\widehat{K}_l^{(fin), i} M_l n_l.
\end{equation*}
In numerical experiments both $\widehat{e}_{10^4}(\mc(\varepsilon))$ and $\widehat{e}_{10^3}(\widehat{\mlmc}(\varepsilon))$ were evaluated for various values of $\varepsilon>0.$ 

In figure \ref{fig:mc} the reader can refer to log-log plot of $\widehat{e}_{10^4}(\mc(\varepsilon))$ vs $\cost(\mc(\varepsilon))$ as well as the expected theoretical slope in terms of $\varepsilon>0$. On the other hand, in figure \ref{fig:mlmc} one can find a plot of $\widehat{e}_{10^3}(\widehat{\mlmc}(\varepsilon))$ vs $\cost(\widehat{\mlmc}(\varepsilon))$ as well as the comparison with expected cost upper bound with unknown constants obtained from nonlinear regression. Note that, both estimated $L^{2}(\Omega)$-error and ifnormational cost are random. Hence, some observations in figure \ref{fig:mlmc} fall above estimated theoretical cost upper bound.
Similarily, in figure \ref{fig:mlmc_err} we compare estimated $L^{2}(\Omega)$-error of $\widehat{\mlmc}(\varepsilon)$ with $\varepsilon >0.$
Finally, in figure \ref{fig:mc_vs_mlmc}, one can find the comparison between informational costs of different algorithms with respect to estimated $L^{2}(\Omega)$-errors.



\begin{figure*}[h!]
    \centering
    \begin{subfigure}[t]{0.5\textwidth}
        \centering
        \includegraphics[width=1.0\textwidth]{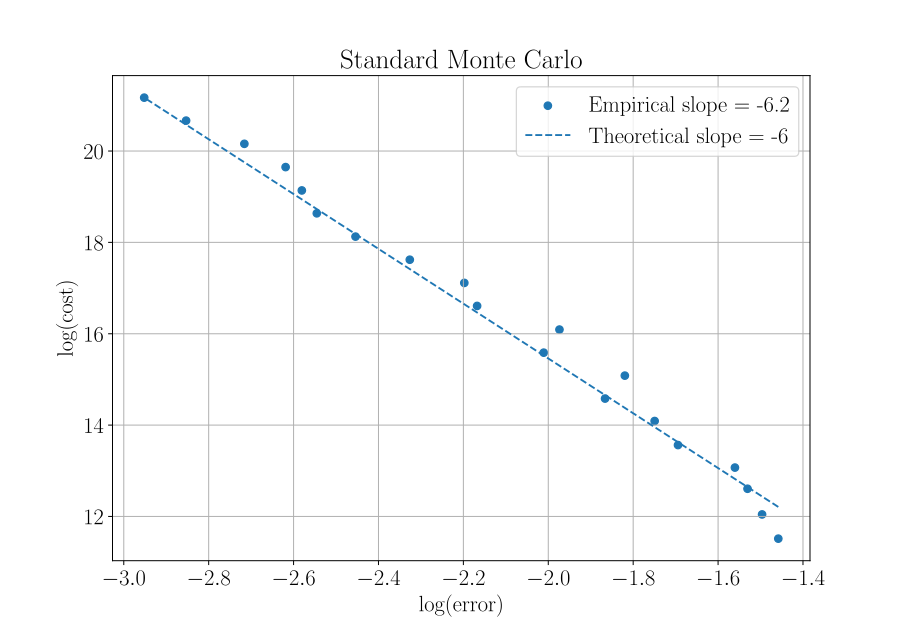}
    \caption{Monte Carlo MSE error vs informational cost}
    \label{fig:mc}
    \end{subfigure}%
    ~ 
    \begin{subfigure}[t]{0.5\textwidth}
        \centering
        \includegraphics[width=1.0\textwidth]{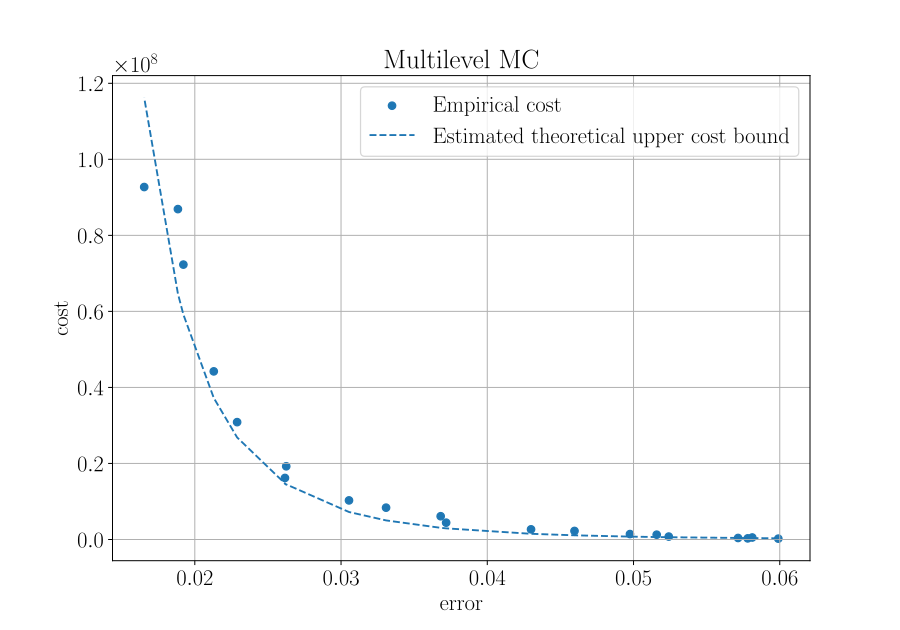}
    \caption{MLMC MSE error vs informational cost}
    \label{fig:mlmc}
    \end{subfigure}
    \caption{Error vs informational cost comparison.}
\end{figure*}


\begin{figure*}[h!]
    \centering
    \begin{subfigure}[t]{0.5\textwidth}
        \includegraphics[width=1.0\textwidth]{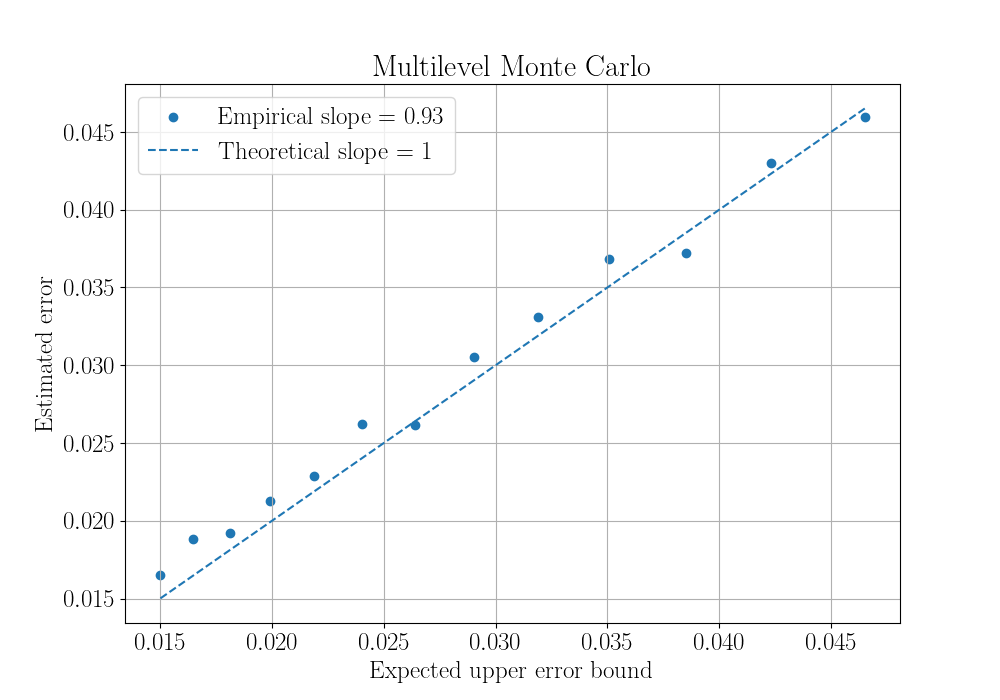}
    \caption{MLMC error and upper error bound}
    \label{fig:mlmc_err}
    \end{subfigure}
    ~ 
    \begin{subfigure}[t]{0.5\textwidth}
        \centering
        \includegraphics[width=1.0\textwidth]{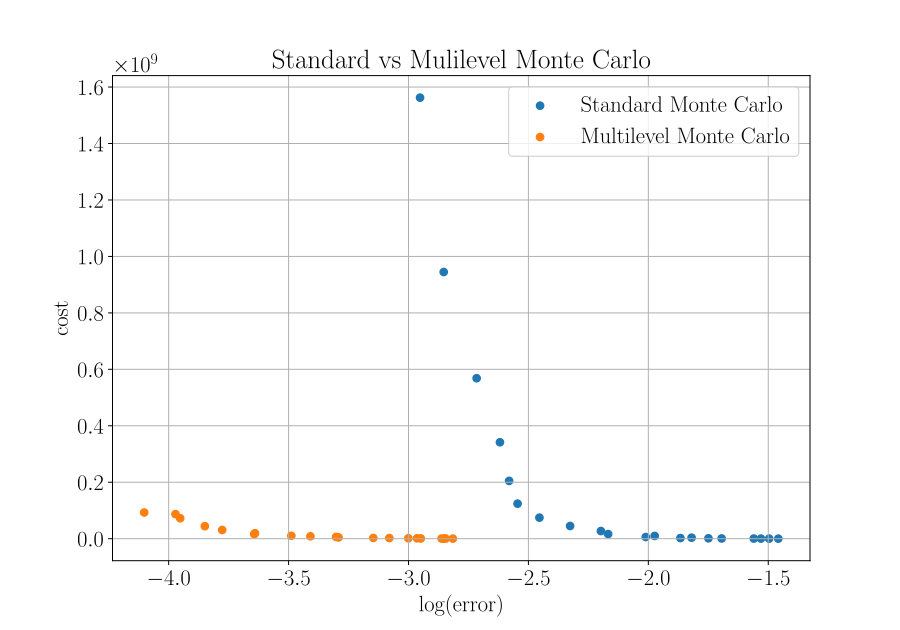}
    \caption{Standard Monte Carlo cost vs MLMC cost }
    \label{fig:mc_vs_mlmc}
    \end{subfigure}%
    \caption{Estimated MLMC error and cost comparison with standard Monte Carlo}
\end{figure*}
\end{example}

\subsection{Details on the implementation}
\label{sec:impl}
For the convenience of the reader, we provide the following code listings that contain implementation of MLMC algorithm.
The implementation utilizes both Python and CUDA C programming languages as well as PyCuda package that allows to call CUDA kernels directly from Python. The pseudo-code for the algorithm that dynamically estimates the number of levels and the variance is similar to the one presented in \cite{GILES}.

A single step of truncated dimension randomized Euler algorithm was implemented as a CUDA device function. See listing 3 in \cite{SIAM} for the reference.

On the very top of the abstraction hierarchy, we implemented the multilevel method in Python that makes direct calls to CUDA kernels via PyCuda API.
The implementation is shown in the listing \ref{mlmccode} below.

\begin{lstlisting}[language=Python, caption=MLMC implementation, basicstyle=\ttfamily\tiny, label=mlmccode]
def run_adaptive(self, x0: float, T: float, M: Callable, n: Callable, eps: float=1e-3) -> Tuple[float, float]:
        # (1)
        L = 0 
        Y, V, N = [], [], [] 
        Yl1, Yl2 = np.inf, np.inf
        beta = n(1)/n(0)
        convergence_err = np.inf
        # (2)
        while convergence_err > 0:
            N.append(10**3)
            # (3)
            YL, VL = self.__get_Yl(level=L, M=M, n=n, N=N[L], x0=x0, T=T)
            Y.append(YL)
            V.append(VL)
            _N = get_N(M, n, V, eps)
            # (4)
            for l in range(L+1):
                if _N[l] <= N[l]:
                    continue
                Yl, Vl = self.__get_Yl(level=l, M=M, n=n, N=(_N[l]-N[l]), x0=x0, T=T)
                Vl = (
                    (Vl + (_N[l]-N[l])/(_N[l]-N[l]-1)*Yl**2) * (_N[l]-N[l]-1) + 
                    (V[l] + N[l]/(N[l]-1)*Y[l]**2) * (N[l]-1)
                )/(_N[l]-1)
                Y[l] = (Yl*(_N[l]-N[l]) + Y[l]*N[l])/(_N[l])
                V[l] = Vl - (_N[l]/(_N[l]-1))*Y[l]**2
                N[l] = _N[l]
            Yl1 = Yl2
            Yl2 = YL
            L += 1
            # (5)
            convergence_err = max(
                abs(Yl2), abs(Yl1) / beta
            ) - (beta**(0.5)-1) * eps/np.sqrt(2)
        # (6)
        levels = np.arange(L)
        return np.sum(Y), np.sum(M(levels) * n(levels) * N)
\end{lstlisting}
\begin{enumerate}
    \item [(1)] Initialize local variables.
    \item [(2)] Run the main loop of the procedure.
    \item [(3)] Estimate expectation and expectation of a squared payoff with direct CUDA kernel call.
    \item [(4)] Update a number of samples, expectations, and variance estimates per level if needed.
    \item [(5)] Calculate convergence error.
    \item [(6)] Return the resulting estimate and the corresponding informational cost of an algorithm.
\end{enumerate}

On the other hand, CUDA C implementation of a coupling on different levels is provided in the listing \ref{ylcode} below.

\begin{lstlisting}[language=C, caption=MLMC implementation, basicstyle=\ttfamily\tiny, label=ylcode]
__device__ FP sample_from_Yl(int ML, int nL,  int Ml, int nl, FP x0, FP T, curandState_t* state) {
    // (1)
    FP* dWL = (FP*)malloc(sizeof(FP) * ML);
    memset(dWL, 0, sizeof(FP) * ML);
    FP* dWl = (FP*)malloc(sizeof(FP) * Ml);
    memset(dWl, 0, sizeof(FP) * Ml);
    // (2)
    FP tL = 0;
    FP tl = 0;
    FP XL = x0;
    FP Xl = x0;
    // (3)
    int grid_density = LCM(nL, nl);
    FP H = T / grid_density;
    // (4)
    Jump* jumps_head = (Jump*)malloc(sizeof(Jump));
    generate_jumps<FP>(state, INTENSITY, T, jumps_head);
    Jump* jump_L = jumps_head;
    Jump* jump_l = jumps_head;
    // (5)
    FP t, dW;
    for (int i = 0; i < grid_density; i++) {
        t = (i+1)*H;
        // (6)
        for (int k = 0; k < ML; k++) {
            dW = (FP)(curand_normal(state) * sqrt(H));
            dWL[k] += dW;
            if (k < Ml) {
                dWl[k] += dW;
            }
        }
        // (7)
        if ((i+1)%(grid_density/nL) == 0) {
            jump_L = first_jump_after_time(jump_L, tL);
            XL = euler_single_step<FP>(tL, XL, t-tL, dWL, ML, jump_L, state);
            tL = t;
            memset(dWL, 0, sizeof(FP) * ML);
        }
        // (8)
        if ((i+1)%(grid_density/nl) == 0) {
            jump_l = first_jump_after_time(jump_l, tl);
            Xl = euler_single_step<FP>(tl, Xl, t-tl, dWl, Ml, jump_l, state);
            tl = t;
            memset(dWl, 0, sizeof(FP) * Ml);
        }
   }
   free_jumps_list(jumps_head);
   free(dWL);
   free(dWl);
   return func_f(XL) - func_f(Xl);
}
\end{lstlisting}
\begin{enumerate}
    \item [(1)] Initialize sparse and dense grid Wiener increments.
    \item [(2)] Initialize temporary variables for sparse and dense grid trajectories.
    \item [(3)] Get least common multiple of grid densities.
    \item [(4)] Generate all jumps.
    \item [(5)] Traverse through grid points.
    \item [(6)] Update Wiener increments.
    \item [(7)] Update dense grid trajectory value.
    \item [(8)] Update sparse grid trajectory value.
\end{enumerate}

\section{Conclusions}
\label{sec:conclude}
In this paper, we analyzed the MLMC method for SDEs driven by countably dimensional Wiener process and Poisson random measure in terms of theory and numerical experiments. The main theorem shows that the MLMC method can be applied to a class of SDEs for which their coefficients satisfy certain regularity conditions including discontinuous drift, Hölder-continuous diffusion and jump-function with Hölder constants greater than or equal to one-half. Under provided complexity model, the resulting informational cost is reduced similarly as in the finite-dimensional case. The resulting thesis coincides with the case that the Wiener process is finite (see \cite{GILES}), meaning the cost reduces from $\Theta(\varepsilon^{-4})$ to $\Theta(\varepsilon^{-2}(\log(\varepsilon))^{2})$. In infinite dimensional case we obtained the reduction from $\Theta(\varepsilon^{-4}\delta^{-1}(\varepsilon))$ to $\bigo(\varepsilon^{-2}(\log(\varepsilon))^{2}\delta^{-1}(\varepsilon^{\kappa}))$ with possibly unknown constant $\kappa>1.$ The impact of the unknown constant can be mitigated if the inverse of $\delta$ belongs to a certain class of functions. 
The MLMC method which depends on the additional set of parameters (truncation dimension parameters) is therefore a natural extension of a multilevel method that depends only on the grid density. On the other hand, the lower cost bound for the multilevel method shows that the cost is always greater than the one obtained for a multilevel method for SDEs driven by the finite-dimensional Wiener process. The lower cost bound is (up to constant) proportional to $\varepsilon^{-2}(\log(\varepsilon))^{2} + \varepsilon^{-2}\delta^{-1}(\varepsilon)$ which is equal to the sum of costs for the evaluation of multilevel method in finite-dimensional setting and the truncated dimension Euler algorithm. It is rather a natural consequence of the combined usage of those two algorithms. 

To conclude, this paper paves the way for further research regarding the following open questions:
\begin{itemize}
    \item Can the unknown constant in the exponent of the cost upper bound of the MLMC method be mitigated in general?
    \item What is the cost upper bound if one of Hölder constants is less than one-half?
    \item What are the worst-case complexity lower bounds?
\end{itemize}

In future research, we plan to investigate the error of the MLMC method
under inexact information for the weak approximation of solutions of SDEs.

\section{Acknowledgments}
I would like to thank my supervisor Paweł Przybyłowicz for guidance and inspiration to work on that topic. I would like to thank the unknown reviewer for valuable comments and suggestions to improve the overall quality of this work.

\printbibliography
\end{document}